\documentclass{article}
\usepackage{amssymb}
\title{  The Continuity of Linear and Sublinear Correspondences Defined on  Cones\\[0.3cm]}

\author{{Masoumeh Aghajani$^1$, \,\,\,Kourosh Nourouzi$^{2}$ \thanks{ Corresponding
author } \thanks {e-mails: nourouzi@kntu.ac.ir, nourouzi@ipm.ir;
fax: +98 21 22853650}
 ,\,\,\, Donal O'Regan$^3$ }\\[0.4cm]
{ \em $^{1,2}$ Department of Mathematics,  K. N. Toosi University of Technology,}\\
{\em P.O. Box 16315-1618, Tehran, Iran.}\\
  \and
{\em $^{3}$ School of Mathematics, Statistics and Applied Mathematics,}\\
{\em
National University of Ireland, Galway, University Road, Galway, Ireland}\\
}

\newenvironment{proof}{\noindent {\em {Proof .}}}{$\square$

\medskip}

\newtheorem{definition}{Definition}
\newtheorem{corol}{Corollary}
\newtheorem{ex}{Example}
\newtheorem{thm}{Theorem}

\newtheorem{lem}{Lemma}

\begin{document}

\maketitle \begin{abstract} In this paper, we investigate the
continuity  of linear and sublinear correspondences defined on
 cones in normed spaces. We also generalize some known results  for sublinear correspondences.
\end{abstract} \maketitle

\renewcommand{\baselinestretch}{1.1}
\def\thefootnote{ \ }

\footnotetext{{\em} $2010$ Mathematics Subject Classification.
Primary: 47A06, Secondary: 54C60
}

\section{Introduction and Preliminaries}
An investigation of linear correspondences defined on  cones in
normed spaces was given  in \cite{ol}. In particular, the
existence of  a unique iteration semigroup  of continuous linear
selections  of  an iteration semigroup of linear correspondences
defined on a cone with a finite cone basis is shown in \cite{ol}.
It is  shown in \cite{sm} that a regular cosine family consisting
of superadditive mappings continuous and homogeneous with respect
to positive rationals with compact values has exponential growth.
 The continuity of  a regular cosine family consisting of
continuous and additive mappings   with compact and convex values
defined on cones with nonempty interior in Banach spaces is
established in \cite{sm}.
 A generalization of these results in normed spaces can  be found  in \cite{kn}.


In this paper we reintroduce linear and sublinear correspondences
on cones in real normed spaces and give some results on
continuity. A general form of linear and sublinear correspondences
with convex and compact values is given. We also present  some
results on  invertibility of selections of sublinear
correspondences and some results for an iteration semigroup of
sublinear correspondences. More precisely, the
outline of this paper is as follows.  
In  Lemma \ref{hembound} we give a necessary and sufficient
condition for upper semicontinuity of a sublinear correspondence.
In Lemma \ref{lemma} we show that the inequality given in Lemma 2
of \cite{ol} can be replaced by  equality. Corollary \ref{1}
and Corollary \ref{2} show that the validity of Lemma 2 and Lemma
3 in \cite{ol} for sublinear correspondences, respectively.
Theorem \ref{thm} is a restatement of Theorem 1 in \cite{ol} for
sublinear correspondences.
%

We begin with some basic concepts which are needed in this paper.

A subset $C $ of a real normed space $X$ is a cone if
$tC\subseteq C$ for every $t>0.$ A linearly independent set $E$
is said to be a basis of cone $C$ if
$$C=\{x\in X: x=\lambda_1 e_1 +\lambda_2 e_2 +\cdots+\lambda_ne_n, n\in N , e_i \in E,\lambda_i\geq0 , i=1,\cdots,n\}.$$
Throughout this paper we assume that $X$ and $Y$ are two real
normed spaces and $C$ is a convex cone of $X$.

Let $c(X)$ denote the set of all nonempty and compact subsets of
$X$ and $cc(X)$ be the family of all convex sets of $c(X)$.

We recall that a  correspondence $\varphi$ on any subset $E$ of
$X$ is a relation which assigns a nonempty set of $Y$ to each
element of $E$. We use the notations $\varphi:C\rightarrow c(Y)$
and $\varphi:C\rightarrow cc(Y)$ for correspondences with compact
values
 and convex and compact values, respectively.
\begin{definition}{\rm
\cite{ol}\label{sub}
 A correspondence  $\varphi:C\twoheadrightarrow
Y$
 is called:
\begin{enumerate}
\item linear if  $\varphi(x+y)=\varphi(x)+\varphi(y)$ (additivity) and
$\varphi(\lambda x)=\lambda \varphi(x)$, for every $x,y\in C$ and
$\lambda>0$; \item sublinear if $\varphi(x+y)\subseteq
\varphi(x)+\varphi(y)$ and $\varphi(\lambda x)=\lambda
\varphi(x)$, for every $x,y\in C$ and $\lambda>0.$
\end{enumerate}
}\end{definition}

 It is clear that every linear correspondence is sublinear
but the converse is not true. 

\begin{definition} {\rm \cite{sm}  A correspondence $\varphi:C\twoheadrightarrow
Y$ is said to be bounded if for every bounded subset $E$ of $C$
the subset $\varphi(E)$ is bounded in $Y$.}
\end{definition}


%


We recall that a neighborhood of a set $A$ is any set $B$ for
which there is an open set $V$ satisfying $A\subseteq V\subseteq
B.$

\begin{definition}{\rm \cite{alip} A correspondence $\varphi:C\twoheadrightarrow
Y$ is said to be: \begin{enumerate} \item  upper semicontinuous at
the point $x$ if for every neighborhood $U$ of $\varphi(x),$ there
is a neighborhood $V$ of $x$ such that $z\in V$ implies
$\varphi(z)\subseteq U$. Also $\varphi$ is upper semicontinuous on
$C,$ if it is upper semicontinuous at every point of $C.$
 \item
lower semicontinuous at the point $x$ if for every open set $U$
that $\varphi(x)\cap U\neq\emptyset$ there is a neighborhood $V$
of $x$ such that $z\in V$ implies $\varphi(z)\cap U\neq\emptyset.$
 $\varphi $ is lower semicontinuous on $C,$ if it is lower
semicontinuous at every point of $C.$ \item continuous at $x$ if
it is both upper and lower semicontinuous at $x.$  It is
continuous if it is continuous at each point of $C$.
\end{enumerate}}\end{definition}

 For each pair of nonempty and
compact subsets $A$ and $B$ of $X,$ the Hausdorff metric
$\mathfrak{h}$ is defined as
$$\mathfrak{h}(A,B)=\max\{{\sup}_{a\in A}d(a,B),{\sup}_{b\in
B}d(b,A)\},$$
 where $d(a,B)={\inf}_{b\in B} \|a-b \|$.

 Every  correspondence with compact values  $\varphi:X\twoheadrightarrow Y$  is continuous if and only if $\varphi:X\rightarrow (c(Y), \mathfrak{h})$ is
 continuous in the sense of a single-valued function (see Theorem 17.15 in \cite{alip}).

%
%
%
%
\section{Continuity of Linear and Sublinear Correspondences }
In this section  we study  the continuity of linear and sublinear
correspondences  defined on cones with a finite basis in real
normed spaces. We start with the following.
\begin{lem} {\rm \cite{sm}} \label{equ} A sublinear correspondence $\varphi:C\twoheadrightarrow Y$ is
bounded if and only if there exists a positive constant $M$ such
that
\begin{equation}
\label{bound} \| \varphi(x)\|:=\sup\{\| y\|:y\in \varphi(x)\}\leq
M \| x\|,\,\,\,\,\,\,\,\,\,\,\,\,\,\,\,\,\,\, (x\in C).
\end{equation}
\end{lem}

Lemma 1 in \cite{ol} gives a necessary condition for upper
semicontinuity of a linear correspondence.
\begin{lem}\label{hembound} Let  $0\in C\subseteq X$. If $\varphi:C\twoheadrightarrow Y$
is a bounded-valued sublinear correspondence,  then $\varphi$ is
upper semicontinuous at zero if and only if $\varphi$ is bounded.
\end{lem}
\begin{proof}
If $\varphi$ is upper semicontinuous at zero, then by an argument
similar  to that in the  proof of  (\cite{ol}, Lemma 1) and Lemma
\ref{equ} we get the boundedness of $\varphi$. Conversely, suppose
that  $ \varphi$ satisfies
   (\ref{bound}) and $U$ is a neighborhood  of $\varphi(0)=\{0\}$. Then, there exists
 $\varepsilon>0$  such that
 $N_{\varepsilon}(0)\subseteq U$. Now for every $z\in
 N_{\frac{\varepsilon}{M}}(0)$ we have $\varphi(z)\subseteq U$, i.e.,  $\varphi$ is upper
semicontinuous at zero.
\end{proof}

We define the norm of a bounded  sublinear correspondence
$\varphi:C\twoheadrightarrow Y$ by $$\| \varphi\|=\inf\{M>0:\|
\varphi(x)\|\leq M\| x\|,x \in C\}.$$

\begin{thm}\label{contin} Let  $E=\{e_1,e_2,...,e_n\}$ be a basis of $C$.
If $\varphi:C\rightarrow c(Y)$ is   linear, then $\varphi$ is continuous.
\end{thm}
\begin{proof}
Let $\sim$ denote the R$\dot{\mbox{a}}$dstr$\ddot{\mbox{o}}$m's
equivalence relation  between pairs of members of  $cc(Y)$ defined by
 $$(A,B)\sim (C,D)\Leftrightarrow A+D=B+C,\,\,\,\,\,\,\,\,\,\,\,\,(A,B\in cc(X))$$ and $[A,B]$ denote the  equivalence class of  $(A,B)$ (see \cite{rad}).
 The set of all equivalence classes
$\Delta$ with the operations
$$[A,B]+[C,D]=[A+C,B+D],$$
$$\lambda [A,B]=[\lambda A,\lambda B]\,\,\,\,\,\,\,\,\,\,\,\,\,\,\,(\lambda\geq 0),$$
$$\lambda [A,B]=[-\lambda B,-\lambda A]\,\,\,\,\,\,\,\,\,\,\,\,\,(\lambda< 0),$$ and
the norm $$\|[A,B]\|:=\mathfrak{h}(A,B),$$ constitute a  real
linear normed space (see \cite{rad}). The function
$f:C\rightarrow\Delta$ defined by
$$f(x)=[\varphi(x),\{0\}],$$ is linear and  can be
extended to a linear operator $\hat{f}:C-C\rightarrow\Delta$ by
$$\hat{f}(x-y)=f(x)-f(y), \,\,\,\,\,\,\,\,\,\,\,(x,y\in C).$$ Since $C-C$ is of  finite dimension, $\hat{f}$
and consequently $f$ are continuous. Let $x_0\in C$ and $(x_n)$
be a sequence of $C$ converging to $x_0$. Then
$$\lim_{n\rightarrow
\infty}\mathfrak{h}(\varphi(x_n),\varphi(x_0))=\lim_{n\rightarrow
\infty}\|f(x_n)-f(x_0)\|=0,$$ that is, $\varphi$ is continuous.
\end{proof}

 Each set of the form $M :=M_1\times ...\times M_n$, where $ M_i \subseteq
{\mathbb{R}}^n \,(i=1,...,n)$ will be called a multimatrix. If
$E=\{e_1,e_2,...,e_n\}$ is a basis of $C$ and
$\varphi:C\twoheadrightarrow C$ is linear, then there exists an
isomorphism $l:C-C\rightarrow {\mathbb{R}}^n$ defined by
\begin{equation}\label{iso} l(\sum_{j=1}^{n}\lambda_j e_j
)=(\lambda_1,...,\lambda_n)^T, \end{equation}
 that maps $C$ onto
${[0,+\infty)}^n$ and $M_{\varphi}:=l(\varphi(e_1))\times ...
\times l(\varphi(e_n))$ is a nonempty convex multimatrix
\cite{ol}.

\begin{corol}\label{lincontin} Let  $E=\{e_1,e_2,...,e_n\}$ be a basis of $C$.
If $\varphi:C\twoheadrightarrow C$ is   a linear correspondence,
then
\begin{equation}\label{form} \varphi(x)=\{l^{-1}Al(x)\}_{A\in M_{\varphi}},
\,\,\,\,\,\,\,\,\ (0\neq x\in C)
\end{equation}
 and $\varphi$ is  lower semicontinuous at every point.
\end{corol}
\begin{proof} If $x=\sum_{j=1}^n \lambda_j e_j\in C,$ we
have $l^{-1}Al(x)={\sum}_{i=1}^{n} \sum_{j=1}^{n}({\lambda}_j
a_{ij})e_i$, for every ${A=[a_{ij}]}\in M_{\varphi}$ and therefore

$$\{{\sum}_{i=1}^{n} \sum_{j=1}^{n}({\lambda}_j a_{ij})e_i
:{A=[a_{ij}]}\in M_{\varphi} \}=\{l^{-1}Al(x)\}_{A\in
M_{\varphi}}.$$ 
Thus, it suffices  to show that  $$\varphi(x)=\{{\sum}_{i=1}^{n}
\sum_{j=1}^{n}({\lambda}_j a_{ij})e_i :{A=[a_{ij}]}\in M_{\varphi}
\} \,\,\,\,\,\,\,\,\,(0\neq x=\sum_{j=1}^{n}{\lambda}_j e_j\in
C).$$
Let ${\varphi_M}(x)$ be the quantity of the right hand and  $z\in
\varphi(x)$. We may find $ u_j\in \varphi(e_j) $ and then $A=[a_{ij}]\in
 M_{\varphi}$  such
that

$$z=\sum_{j=1}^n \lambda_j u_j =\sum_{j=1}^n \lambda_j \sum_{i=1}^{n} ( a_{ij}e_i)=
\sum_{j=1}^n \sum_{i=1}^n ({\lambda}_j a_{ij})e_i.$$
 Therefore $
z\in {\varphi_M}(x).$ It is easy to see that $ \varphi_M(x)
\subseteq \varphi(x)$ and therefore $\varphi(x)={\varphi}_M(x)$
for every $x\in C\setminus\{0\}.$ To see the lower semicontinuity
of $\varphi,$ let $x\in C\setminus\{0\}$ and $U$ be an open set
with $U\cap\varphi(x)\neq\emptyset.$ From (\ref{form}) and the
continuity of $l^{-1}Al$ for each $A\in M_\varphi$ there exists an
open neighborhood $V$ of $x$ such that  $\varphi(z)\cap
U\neq\emptyset$, for each $z\in V.$ Now to see the lower
semicontinuity of $\varphi$ at zero, let $(x_n)_n$ be a sequence
convergent to zero and $y\in \varphi(0).$ We can assume that all
$x_n$'s are nonzero. Fix an $A\in M_\varphi$ so
$\lim_{n\rightarrow\infty} l^{-1}Al(x_{{n}})= 0$ and the sequence
$(z_n)$ with
$${z_{n}}=y+l^{-1}Al(x_{{n}})\in
\varphi(0)+\varphi(x_{{n}})=\varphi(x_{{n}}),$$  tends to $y$.
Thus, by Theorem 17.21 in \cite{alip}, the proof is complete.
\end{proof}

 In Corollary
\ref{lincontin} if $\varphi:C\rightarrow c(C)$, then $\varphi$ is
continuous, by Theorem \ref{contin}.


Let $C$ be a cone with a finite basis $\{e_1,e_2,...,e_n\}$. For
every sublinear correspondence $\varphi:C\rightarrow c(Y)$ there
exists a linear continuous correspondence
$\widehat{\varphi}:C\rightarrow cc(Y_0)$ containing  $\varphi$,
defined by
\begin{equation}\label{tild}
\widehat{\varphi}(x)=\sum_{j=1}^n {\lambda_j}
\overline{co}(\varphi(e_j)),
\end{equation}
for every $x=\sum_{j=1}^n \lambda_j e_j$,  where $Y_0$ and
$\overline{co}(\varphi(e_j))$ denote the completion of $Y$ and the closed convex hull of
the set $\varphi(e_j)$ in $Y_0$, respectively.


\begin{corol}\label{subcontin}
 Let $E=\{e_1,e_2,...,e_n\}$ be a basis of $C$. If $\varphi:C\rightarrow c(Y)$ is a sublinear
 correspondence, then\\
i) $\varphi$  is  upper semicontinuous at every point;\\
 ii) moreover, if
$\varphi:C\twoheadrightarrow C$, then for every $x\in
C\setminus\{0\}$ we have
$$\varphi(x)\subseteq (l^{-1}Al(x))_{A\in M_{co(\varphi)}},$$ where $l$ is
the isomorphism given  in (\ref{iso}).

\end{corol}
\begin{proof}
i) Let $Y_0$ be the completion of $Y.$ Consider
$\widehat{\varphi}$ as given in (\ref{tild}). Obviously,
$\widehat{\varphi}:C\twoheadrightarrow Y_0$ is a linear
correspondence with convex and compact values. From Theorem
\ref{contin}, since $\widehat{\varphi}$ is upper semicontinuous at
zero and $\|\varphi\|\leq\|\widehat{\varphi}\|$, Lemma
\ref{hembound} implies that
$\varphi$ is upper semicontinuous at zero. Now let $x_0$ be a
nonzero element in $C$. For any neighborhood $U$ of
$\varphi(x_0)$, there exists an open ball $U_0$ centered  at zero
such that,
$$\varphi(x_0)-U_0+U_0\subseteq U.$$ Since $\varphi$ is upper semicontinuous
at zero then there is an open  neighborhood $V_0$ of zero in $C$  such that
$\varphi(z)\subseteq U_0$ for every $z\in V_0$. We may suppose
that $\alpha x_0\in V_0$, for some $\alpha\in(0,1)$. Now there is
an open ball $W_0$ of $\alpha x_0$ such that $W_0\subseteq V_0$
with $x_0\notin \overline{W_0}$. Putting $W_{x_0}=x_0-\alpha x_0+W_0$ we see
that $W_{x_0}$ is open. Since $x_0 =\Sigma_{i=1}^n \lambda_i
e_i \notin \overline{W_0}$ there exist $r>0$ and open balls $N_{r}^C (x_0)$
and $N_{r}^C(\alpha x_0)$ in $C$ such that $N_{r}^C(\alpha
x_0)\subseteq W_0$, $N_{r}^C (x_0)\cap W_0=\emptyset$ and
$r<\alpha{\lambda_i}$ for $i=1,\cdots,n$ with $\lambda_i\neq0$. We now show
that
$$N_{r}^C (x_0)\subseteq N_{r}^C(\alpha x_0)+(1-\alpha)x_0.$$
Without loss of generality we assume that
$\|e_1\|=\|e_2\|=\cdots=\|e_n\|=1$. Let $z=\Sigma_{i=1}^n \mu_i e_i
\in N_{r}^C (x_0)$, so $-r< \mu_i - \lambda_i <r$ for each $1\leq i\leq n$
with $\lambda_i\neq0$. Therefore, $$0<-r + \alpha \lambda_i < \mu_i - \lambda_i+\alpha
\lambda_i <r+ \alpha \lambda_i
\,\,\,\,\,\,\,\,\,\,\,(i=1,\cdots,n),\,\,\,\,\lambda_i\neq0,$$ and
$$0<\mu_i<r \,\,\,\,\,\,\,\,\,\,\,(i=1,\cdots,n),\,\,\,\,\,\,\lambda_i=0.$$ Thus
$\Sigma_{i=1}^n (\mu_i -(1-\alpha)\lambda_i)e_i \in C$. Since
$$\|z-(1-\alpha)x_0 - \alpha x_0\|=\|z-x_0\|<r,$$ so $$N_{r}^C (x_0)-(1-\alpha)x_0\subseteq N_{r}^C(\alpha
x_0).$$

 Now for every $z\in W_{x_0}$,
there is a point $z_0\in W_0$ such that $z=x_0-\alpha x_0+z_0$ and from above (note $\alpha x_0, z_0 \in V_0$)

\begin{center}
\begin{tabular}{lll}
$\varphi(z)=\varphi(x_0-\alpha x_0+z_0)$&$\subseteq$& $(1-\alpha)\varphi(x_0)+\varphi(z_0)$\\[0.2cm]
&$\subseteq$&$ \varphi(x_0)-\alpha\varphi(x_0)+\varphi(z_0) $\\[0.2cm]
&$\subseteq$&$\varphi(x_0)-U_0 +U_0$\\[0.2cm]

&$\subseteq$& $U.$
\end{tabular}
\end{center}
Thus $\varphi$ is upper semicontinuous at $x_0$.\\
ii)  Consider the linear correspondence $\widehat{\varphi}$ as (\ref{tild}).
 By Corollary \ref{lincontin},
$\widehat{\varphi}$ is of the form (\ref{form}). Since  $M_{co(\varphi)}=M_{\widehat{\varphi}}$  and
$\varphi(x)\subseteq\widehat{\varphi}(x)$ for each nonzero $x$, we
obtain the desired inclusion.
\end{proof}

The following example shows that a sublinear correspondence need
not be lower semicontinuous at every point.

\begin{ex} {\rm Define
$\varphi:[0,+\infty)\times[0,+\infty)\rightarrow
[0,+\infty)\times[0,+\infty)$ by

$$\varphi(x,y)=\left\{%
\begin{array}{ll}
    \{(0,0)\} &  {x\geq 0 , y>0 ;} \\
  \{(t,0):0\leq t\leq x\} & {x\geq 0 , y=0.} \\
\end{array}%
\right.$$
 }\end{ex}

 It is easy to see that the sublinear
correspondence $\varphi$ is not lower semicontinuous at every
point $(x,0)$ where $x>0$.

For the rest of this section we  consider, inspired by \cite{ol},
the relations between Hausdorff distance of the unit matrix and
multimatrix of a linear correspondence and invertibility of its
selections.

 Every cone $C$ with a finite
basis $E=\{e_1,...,e_n\}$ induces a norm  on the vector space of
all $n\times n$ matrices $\mathbb{M}_n (\mathbb{R})$ by
\begin{equation}
\label{norm}\,\,\,  \| A\|=\sup\{{\ \|{\sum_{i=1}^n (\sum_{j=1}^n
{{\lambda_j} a_{ij}})e_i}\|}: {\sum_{j=1}^n {{\lambda_j}
e_{j}}}\in C, \|\sum_{j=1}^n {{\lambda}_j} e_{j}\|=1\},
\end{equation}
for every $A=[a_{ij}]$ (see \cite{ol}).

In the following, $\mathfrak{h}_1$ and $\mathbb{I}$ will denote
the Hausdorff metric derived from the norm given  in (\ref{norm})
and the unit matrix, respectively.
\begin{lem} \label{lemma} Suppose that $C$ has a finite cone basis.
If $\varphi:C\rightarrow c(C)$ is a linear correspondence, then
$$\mathfrak{h}_1(M_{\varphi},\{\mathbb{I}\})=\sup\{\mathfrak{h}(\varphi(x),\{x\}):x\in C,\|
x\|=1\}.$$
\end{lem}
\begin{proof}
 From Lemma 2 in \cite{ol}, we have $$\mathfrak{h}_1(M_{\varphi},\{\mathbb{I}\})\leq \sup\{\mathfrak{h}(\varphi(x),\{x\}):x\in C,\|
x\|=1\},$$ and from  Corollary \ref{lincontin},
$\varphi(x)={\{l^{-1}Al(x)\}}_{A\in M_{\varphi}}$ for each $x\in
C\setminus\{0\}.$ If $x\in C$ with $\|x\|=1$, then by Lemma 3.76
in \cite{alip} there exists ${A_x}\in {M_{\varphi}}$  such that
$$\mathfrak{h}(\varphi(x),\{x\})=\|l^{-1}{A}_x l(x)-x\|\leq \|A_x
-\mathbb{I}\|.$$ Therefore
\begin{center}
\begin{tabular}{lll}
$\sup\{\mathfrak{h}(\varphi(x),\{x\}):x\in C,\|
x\|=1\}$&$\leq$& $\sup\{ \|A_x -\mathbb{I}\|:x\in C,\|x\|=1\} $\\[0.2cm]
&$\leq$&$\sup\{ {\|A -\mathbb{I}\|}:A\in {M_{\varphi}}\}$\\[0.2cm]
&=&$\mathfrak{h}_1 (M_{\varphi},\{\mathbb{I}\}).$
\end{tabular}
\end{center}
\end{proof}

%
%
\begin{corol} \label{1} Suppose that $C$ has a finite cone basis. If $\varphi:C\rightarrow c(C)$ is a
sublinear correspondence, then
$$\mathfrak{h}_1(M_{\widehat{\varphi}},\{\mathbb{I}\})\geq
\sup\{\mathfrak{h}(\varphi(x),\{x\}):x\in C,\|x\|=1\},$$ where
$\widehat{\varphi}$ is given in (\ref{tild}).
\end{corol}
\begin{proof} The proof is an easy application of Lemma \ref{lemma} and
$\varphi(x)\subseteq\widehat{\varphi}(x)$ for each $x\in C$.
\end{proof}

 In \cite{ol}, it is shown that for a cone $C$ with a finite basis  there exists an $\eta>0$ such that for every
linear correspondence $\varphi:C\rightarrow c(C)$ satisfying
$$\mathfrak{h}_1(M_{\varphi},\{\mathbb{I}\})<\eta,$$ every $A\in M_{\varphi}$ is invertible. If
$co\varphi:C\rightarrow cc(C)$ denotes the correspondence
$x\rightarrow co\varphi(x)$, then we have the following result.

\begin{corol}\label{2} Let $\{e_1,e_2,\cdots,e_n\}$  be a finite basis of $C$.
Then, there exists an $\eta>0$ such that for every sublinear
correspondence $\varphi:C\rightarrow c(C)$ satisfying
$\mathfrak{h}_1(M_{co(\varphi)},\{\mathbb{I}\})<\eta$, each $A\in
M_{\varphi}$ is invertible.
\end{corol}
\begin{proof} Consider ${\widehat{\varphi}}$ as given in (\ref{tild}). Since $\widehat{\varphi}$ is linear
 with convex and compact values, by Lemma 3 in
{\cite{ol}}, there exists $\eta>0$ such that for every linear
correspondence $\widehat{\varphi}$ with $
\mathfrak{h}_1(M_{\widehat{\varphi}},\{\mathbb{I}\})<\eta$, then
$A\in {M_{\widehat{\varphi}}}$ is invertible. Since
$M_{co(\varphi)} = M_{\widehat{\varphi}}$ and
$M_{\varphi}\subseteq M_{co(\varphi)}$, every $A\in M_{\varphi}$
is invertible.
\end{proof}


\section{Iteration Semigroups of Sublinear Correspondences}

In this section we investigate the  continuity of an iteration
semigroup of sublinear correspondences.  Theorem \ref{thm} is, in
fact, a  generalization of  Theorem 1 in \cite{ol}.

Recall that the composition of two correspondences $\varphi:
X\twoheadrightarrow Y$ and $\psi:Y\twoheadrightarrow Z$ is defined
by
$$\psi\circ\varphi(x)=\cup_{y\in\varphi(x)}\psi(y),\,\,\,\,\,\,\,\,\,\,(x\in
X).$$




\begin{definition} {\rm \cite{ol}  A family $\{\varphi^t:t\geq0\}$ of
correspondences  $\varphi^t:C\twoheadrightarrow C$ is called an
iteration semigroup if  $\varphi^t\circ \varphi^s=\varphi^{t+s}$
for all $t,s\geq0.$ An iteration semigroup $\{\varphi^t:t\geq0\}$
of correspondences  $\varphi^t:C\rightarrow cc(C)$ is said to be
continuous if for every $x\in C$ the correspondence $t\rightarrow
\varphi^t(x)$ is continuous.}
\end{definition}
%



\begin{lem}{\rm \cite{sm}}\label{linuni}
Let $C$ be convex  with nonempty interior. Then there exists $M>0$
such that for every linear continuous correspondence
$\varphi:C\rightarrow c(Y)$ the inequality
$$\mathfrak{h}(\varphi(x),\varphi(y))\leq M \|\varphi\| \,\|x-y\|,\,\,\,\,\,\,\,\,\,(x,y\in C)$$ holds.
\end{lem}
As a direct  result of Lemma \ref{linuni}, we get the following result.\\
If $E=\{e_1,e_2,...,e_n\}$  is a finite basis of $C$, then there
is $M>0$ such that for every compact-valued sublinear
correspondence $\varphi,\, \psi:C\rightarrow c(C),$
$$\mathfrak{h}(\widehat{\varphi}\circ\widehat{\psi}(x),\widehat{\varphi}(x))\leq M\|
\widehat{\varphi}\|\,
\mathfrak{h}(\widehat{\psi}(x),\{x\}),\,\,\,\,\,\,\,\, (x\in C)$$
where $\widehat{\varphi}$ is of the form (\ref{tild}).
\begin{lem}\label{sublinuni} Let $C$ be a cone with finite basis
$\{e_1,e_2,...,e_n\}$. If $\{\varphi^t:C \rightarrow
cc(C)\}_{t\geq0}$ is an iteration semigroup of sublinear
correspondences $\varphi^t$ with $\varphi^0(x)=\{x\}$, then there
exists $M>1$ such that
\begin{equation}\label{tilda}
\mathfrak{h}(\widehat{\varphi^{w+s}}(x),\widehat{\varphi^{w}}(x))\leq
M\,\|
\widehat{\varphi^{w}}\|\,\|\widehat{\varphi^s}-\varphi^0\|\,\|x\|,
\end{equation}
for each $w,s\geq0$ and $x\in C$.
\end{lem}
\begin{proof} For every $w\geq0,s\geq0$ and
$z\in\widehat{\varphi^{s+w}}(x)\subseteq\widehat{\varphi^{w}}\circ
\widehat{\varphi^{s}}(x)$, we get
$$d(z,\widehat{\varphi^{w}}(x))\leq
\mathfrak{h}(\widehat{\varphi^{w}}\circ\widehat{\varphi^{s}}(x),\widehat{\varphi^{w}}(x)).$$
Therefore by Lemma \ref{linuni},
 $$d(z,\widehat{\varphi^{w}}(x))\leq M\|
\widehat{\varphi^{w}}\|\,\mathfrak{h}(\widehat{\varphi^{s}}(x),\{x\}),$$
and so
\begin{equation}\label{aval}
\sup_{z\in \widehat{\varphi^{w+s}}(x)}
d(z,\widehat{\varphi^{w}}(x))\leq M\|
\widehat{\varphi^{w}}\|\,\|\widehat{\varphi^{s}}-\varphi^0\|\,\|x\|.
\end{equation}
Without loss of generality we can assume that $M>1$. On other hand
for  $z\in \widehat{\varphi^{w}}(x)$,  there exist $z_i\in \varphi^w(e_i)$, $i=1,\dots,n$ such that $z=\sum_{i=1}^n \lambda_i e_i$. Thus
\begin{center}
\begin{tabular}{lll}
$\widehat{\varphi^{w+s}}(x)$&$=$& $\sum_{i=1}^n \lambda_i \varphi^s \varphi^w (e_i)$\\[0.2cm]
&$\supseteq$&$\sum_{i=1}^n \lambda_i \varphi^s (z_i)$\\[0.2cm]
&$\supseteq$&
$\varphi^s(\sum_{i=1}^n \lambda_i z_i),$
\end{tabular}
\end{center}
and consequently
\begin{center}
\begin{tabular}{lll}
$d(z,\widehat{\varphi^{w+s}}(x))$&$\leq$& $d(z,{\varphi^{s}}(z))$\\[0.2cm]
&$\leq$& $\sup_{y\in \widehat{\varphi^{s}}(z)} \|z-y\|$\\[0.2cm]
&$=$& $\|(\widehat{\varphi^{s}} - \varphi^0) (z)\|$\\[0.2cm]
&$\leq$&$\|z\|\,\|\widehat{\varphi^{s}}-\varphi^0\|$\\[0.2cm]
&$\leq$&
$\|\widehat{\varphi^{w}}\|\,\|\widehat{\varphi^{s}}-\varphi^0\|\,\|x\|.$
\end{tabular}
\end{center}
Hence  \begin{equation}\label{dovom} \sup_{z\in
\widehat{\varphi^{w}}(x)}d(z,\widehat{\varphi^{w+s}}(x))\leq M\|
\widehat{\varphi^{w}}\|\,
\|\widehat{\varphi^s}-\varphi^0\|\,\|x\|.
\end{equation}
Now, (\ref{aval}) and (\ref{dovom}) imply (\ref{tilda}).
\end{proof}

\begin{thm}\label{thm} Let $C$ be a cone with   finite basis $\{e_1, \cdots ,
e_n\}$ and let $B$ be a bounded subset of $C.$ If
$\{{\varphi}^t:C\rightarrow cc(C)\}_ {t\geq 0}$ is an iteration
semigroup of sublinear correspondences satisfying the conditions:
\\i) ${\varphi}^0(x)=\{x\},$ for all $x\in C$;\\
ii) $\lim_{t\rightarrow0}\| {\varphi}^t-{\varphi}^0\|=0;$\\
 then, there
exists $\beta_0>0$ and $\gamma>0$ such that  $\|
{\varphi}^t\|\leq \beta_0 e^{\gamma t}$, for each $t\geq0$ and
\begin{equation} \label{hokm} \forall w \geq 0\,
\forall\varepsilon>0 \, \exists \delta>0 \, \forall x\in B \,(\mid
s-w\mid<\delta\Rightarrow
\mathfrak{h}(\widehat{\varphi^{w}}(x),\widehat{{\varphi}^s}(x))<\varepsilon).
\end{equation}
In particular, $\{\widehat{\varphi^{t}}:t\geq0\}$ is continuous
where $\widehat{\varphi^{t}}$ is of the form (\ref{tild}).
\end{thm}
\begin{proof} We assume that $\|e_1\|=\cdots=\|e_n\|=1.$ Let $\{{\varphi}^t:\,t\geq0\}$ be an iteration semigroup satisfying i) and ii).
Consider $\widehat{\varphi^t}$ as in (\ref{tild}) and
$\|\cdot\|_0$ the norm induced by the  basis $\{e_1, \cdots ,
e_n\}$ of $C-C$ with $\|x\|_0= \sum_{i=1}^n \mid \lambda_i\mid$,
where $x=\sum_{i=1}^n \lambda_ie_i$. Since
\begin{center}
\begin{tabular}{lll}
$\|\widehat{\varphi^t}-\varphi^0 \|$&$=$& ${\sup}\{\|\widehat{\varphi^t}(x)-x \|:\,\|x\|=1,x\in C\}$\\[0.2cm]
&$=$&${\sup}\{\|{\sum}_{i=1}^{n} {\lambda}_i {\varphi^t}(e_i)-{\sum}_{i=1}^{n} {\lambda}_i e_i\|:\,\|x\|=1,x={\sum}_{i=1}^{n} {\lambda}_i e_i\} $\\[0.2cm]
&$\leq$&$\sup\{{\sum}_{i=1}^{n}{\lambda}_i
\|{\varphi^t}(e_i)-e_i\|:\,\|x\|=1,x={\sum}_{i=1}^{n} {\lambda}_i e_i\}$\\[0.2cm]
&$\leq$& $k\|{\varphi^t}-\varphi^{0}\|,$
\end{tabular}
\end{center}
for some $k>0$ where $ \|x\|_0\leq k\|x\|$, for all $x\in C-C$.
From (ii), we have ${\lim}_{t\rightarrow
0}\|\widehat{\varphi^t}-\varphi^0 \|=0.$
 Therefore there exist $\alpha>0$ and $\beta>1$
such that $\|\widehat{\varphi^t}\|\leq \beta$, for all
$t\in[0,\alpha].$
 According to the equality $\varphi^{t+s}=\varphi^t \circ\varphi^s$ and Corollary 1
 in \cite{sm} we have  $$\|\varphi^{t+s}\| \leq
\|\varphi^t\|\,\|\varphi^s\|.$$

Putting  $t=r\cdot \alpha+\delta$, where $0\leq\delta<\alpha$ and
$r$ is a nonnegative integer we obtain
\begin{center}
\begin{tabular}{lll}
$\|\widehat{\varphi^t}\|$&$=$& $\sup\{\|{\sum}_{i=1}^n \lambda_i{\varphi}^{r\alpha +\delta}(e_i)\|:x=\sum_{i=1}^n \lambda_i e_i,\|x\|=1\}$\\[0.2cm]
&$\leq$&$\sup\{{\sum}_{i=1}^n \lambda_i \|{\varphi}^{r\alpha  + \delta}(e_i)\|: x=\sum_{i=1}^n \lambda_i e_i,\|x\|=1\}$\\[0.2cm]
&$\leq$&$k\,\|{{\varphi}^{r\alpha +\delta}}\|.$

\end{tabular}
\end{center}
Therefore for $t\geq0$, we have (note $t=r\alpha+\delta$ and so
$r<\frac{t}{\alpha}$)

\begin{center}
\begin{tabular}{lll}
$\|\widehat{\varphi^t}\|$&$\leq$& $k\,{\|\varphi^\alpha\|}^r\,\|\varphi^\delta\|$\\[0.2cm]
&$\leq$&$k\beta^{r+1}$\\[0.2cm]
&$=$& $k \beta\,\beta^r$\\[0.2cm]
&$\leq$& $\beta_0 e^{\gamma t},$
\end{tabular}
\end{center}
where $\gamma:=\frac{1}{\alpha}\ln \beta$ and $\beta_0=k\beta$.

 Now we will show that (\ref{hokm}) can be established by an argument similar to that in the proof of
  Theorem 1 in \cite{ol}. Let $w>0$ and $B$ be a bounded set.
 By Lemma \ref{sublinuni}, there exists $\rho>1$ such
that
 for each $x\in B$ and $s\geq0$
\begin{center}
\begin{tabular}{lll}
$\mathfrak{h}(\widehat{{\varphi}^{s+w }}
(x),\widehat{\varphi^w}(x))$
&$\leq$&$\rho\|\widehat{\varphi^w}\|\, \|\widehat{{\varphi}^s}-\varphi^{0}\|\,\|x\|$\\[0.2cm]
&$\leq$&$\rho \beta_0 e^{\gamma{w}}\, \|\widehat{{\varphi}^s}- \varphi^{0}\|\,\|x\|$\\[0.2cm]
&$\leq$&$  \rho \beta_0 e^{\gamma w}\|\widehat{{\varphi}^s}
-{\varphi}^0\|\,\| B\|.$
\end{tabular}
\end{center}
On the other hand for every $x\in B$ and $w \geq s\geq0,$
\begin{center}
\begin{tabular}{lll}
$\mathfrak{h}(\widehat{{\varphi}^{w}}(x),\widehat{{\varphi}^{{w}-s}}(x))$
&$\leq$&$\rho\|\widehat{{\varphi}^{{w}-s}}\|\,\|\widehat{{\varphi}^s}-\varphi^0\|\,\|x\|$\\[0.2cm]
&$\leq$&$\rho \beta_0 e^{\gamma({w}-s)}\, \|\widehat{{\varphi}^s}-\varphi^0\|\,\|x\|$\\[0.2cm]
&$\leq$&$  \rho \beta_0 e^{\gamma w}\|\widehat{{\varphi}^s}
-{\varphi}^0\|\,\| B\|.$
\end{tabular}
\end{center}
Now, by (ii),  statement  (\ref{hokm}) holds. Finally for every
$x\in C$, putting  $B=\{x\}$, we get $\lim_{s\rightarrow w}
\mathfrak{h}(\widehat{\varphi^s}(x),\widehat{\varphi^w}(x))=0$ and
$\{\widehat{{\varphi}^{t}}:t\geq0\}$ is continuous.
\end{proof}

\begin{ex}\rm{ Let $C$ be a cone
with a finite basis. Then for the iteration semigroup
$\{\varphi^t:t\geq0\}$ of sublinear correspondences
$\varphi^t:C\rightarrow cc(C)$ given by
$$\varphi^t(x)=[e^{\frac{t}{2}},e^t]x\,\,\,\,\,\,\,(x\in C),$$
we have
$$\|\varphi^t\|\leq e^t,$$ and $$\lim_{s\rightarrow w} \mathfrak{h}(\varphi^s(x),\varphi^w(x))=0,$$ that is the given
family and therefore the family of their linear extensions are
continuous. }
\end{ex}

 
\end{document}